\documentclass[english]{article}
\usepackage{amsmath,amsfonts,amssymb,amsthm,amscd}
\usepackage{stmaryrd}
\usepackage{enumerate}
\usepackage{graphicx}
\usepackage{tikz}

\input{comment.sty}
\includecomment{MM}
\includecomment{CL}

\renewcommand{\a}{\alpha}
\renewcommand{\b}{\beta}
\newcommand{\g}{\gamma}
\newcommand{\G}{\Gamma}
\renewcommand{\d}{\delta}
\newcommand{\D}{\Delta}

\renewcommand{\L}{\Lambda}
\newcommand{\m}{\mu}
\newcommand{\n}{\nu}
\renewcommand{\o}{\omega}
\renewcommand{\O}{\Omega}

\newcommand{\s}{\sigma}

\renewcommand{\t}{\tau}

\newcommand{\e}{\varepsilon}
\newcommand{\f}{\varphi}
\newcommand{\F}{\Phi}
\newcommand{\x}{\xi}

\newcommand{\Ac}{{\mathcal A}}

\newcommand{\E}{{\mathbb E}}
\newcommand{\N}{{\mathbb N}}
\newcommand{\R}{{\mathbb R}}
\newcommand{\T}{{\mathbb T}}
\newcommand{\Z}{{\mathbb Z}}

\newcommand{\Fc}{{\mathcal F}}

\newcommand{\Ic}{{\mathcal I}}
\newcommand{\Gc}{{\mathcal G}}

\renewcommand{\Pr}{{\mathbb P}}

\newcommand{\Aff}{{\rm Aff}}
\newcommand{\Hor}{{\rm Hor}}
\newcommand{\DA}{{\rm DA}}
\newcommand{\Aut}{{\rm Aut}}

\newcommand{\TSP}{{\rm TSP}}

\newcommand{\supp}{{\rm supp}}
\newcommand{\proj}{{\sf proj}}

\newtheorem{theorem}{Theorem}[section]
\newtheorem{proposition}[theorem]{Proposition}
\newtheorem{lemma}[theorem]{Lemma}

\newtheorem{fact}[theorem]{Fact}

\theoremstyle{definition}

\newtheorem{remark}[theorem]{Remark}

\theoremstyle{remark}

\begin{document}

\title{Random walks on the discrete affine group}
\author{J\'er\'emie Brieussel, Ryokichi Tanaka and Tianyi Zheng}
\date{\today}

\maketitle

\begin{abstract}
We introduce the discrete affine group of a regular tree as a finitely generated subgroup of the affine group. We describe the Poisson boundary of random walks on it as a space of configurations. We compute isoperimetric profile and Hilbert compression exponent of the group. We also discuss metric relationship with some lamplighter groups and lamplighter graphs.
\end{abstract}

\section{Introduction}

The affine group $\Aff(\T)$ of an infinite $q+1$-regular tree $\T$ is the group of automorphisms that fix a given end. Random walks with respect to spread-out measures on this locally compact group have been extensively studied by Cartwright, Kaimanovich and Woess \cite{CKW} and Brofferio \cite{BroRenewal}. In particular, they identified the Poisson boundary with the end boundary of the tree, endowed with a limiting distribution.
The affine group of the tree contains many interesting finitely generated subgroups, such as solvable Baumslag-Solitar groups, lamplighter groups over the integers and of course any group acting on a rooted tree, e.g.\ automata groups. In fact, the affine tree can be pictured as a rooted tree where the root has been sent to the boundary.

In this view point, we consider a $q+1$-regular tree with a distinguished infinite geodesic ray from a vertex to an end. If we remove the chosen vertex, we are left with $q+1$ subtrees which are all isomorphic. Let us fix some canonical identification between these subtrees. We define the {\it discrete affine group}  $\DA(\T)$ as the subgroup of $\Aff(\T)$ generated by  all permutations of the $q$ subtrees not containing the ray together with the ``shift'' along the ray --- see Section \ref{dag} for detailed definition. This group resembles  the ``mother'' automata groups  which are generated by rooted permutations together with other ``spreading action'' generators \cite{AmirViragDeg3}.

The present paper studies basic algebraic properties of the group $\DA(\T)$, the metric characteristics and random walks on it. The main result is a description of its Poisson boundary  with respect to any finitely supported probability measure.


The Poisson boundary of a group endowed with a probability measure is a measure space giving description of the bounded harmonic functions via a Poisson formula \cite{KV}. When the Poisson boundary is trivial (i.e.\ a point), then all bounded harmonic functions are constant, which is called the Liouville property. An abstract description of the Poisson boundary is always available. However giving explicit descriptions for various classical groups has been an extensive subject of study over the last decades.

As Kaimanovich has pointed out in \cite[Introduction]{KaiF}, known examples of explicit non-trivial Poisson boundary essentially fall into two categories: either they are geometric boundaries related to hyperbolicity or they are spaces of configuration obtained by stabilization along some sample path. 

In the first category, natural geometric boundaries are identified with 
the Poisson boundaries of,
among others, word hyperbolic groups, some Lie groups and their lattices \cite{Khyp},
where Kaimanovich's Ray and Strip criteria can be applied.
These criteria unify many descriptions studied before in (among others) \cite{Furst}, \cite{Led}, \cite{Der} as well as \cite{K} for affine groups over the reals and \cite{CKW} for affine group over local fields and $\Aff(\T)$. 
They also apply to some dense subgroups in Lie groups \cite{BroAff}, \cite{BroSch}.

In the second category, key examples are lamplighter groups. The space of functions from the base group to the lamp group (space of configurations) is known since \cite{KV} to be a non-trivial quotient of the Poisson boundary, provided the random walk on the base group is transient. Kaimanovich and Vershik conjectured that this space of configurations would be the entire Poisson boundary. This conjecture was first proved true when the base group is free of rank at least 2 by Karlsson and Woess \cite{KW}, and when it is abelian of rank at least 5 by Erschler \cite{Ers}, both using Strip criterion. The conjecture was finally proved true in general by Lyons and Peres \cite{LP}. 
For this purpose, they strengthened Kaimanovich's criteria and established a stronger criterion which we call the {\it Trap criterion}, stated below as Theorem \ref{Trap}.

For many groups defined in terms of their action, one can find a space of configurations which stabilize along the sample path to ensure non-triviality of the Poisson boundary. This is the case for some groups of intermediate growth \cite{ErschlerSubexp}, for higher degree ($d\ge 3$) mother automata groups \cite{AmirViragDeg3}, for higher rank groups of interval exchange maps \cite{JMBMdlS} and for Thompson's group $F$ \cite{KaiF}. This type of quotient of the Poisson boundary is sometimes referred to as ``the lamplighter boundary", see \cite[Remark 5]{ErschlerSubexp}. In the aforementioned examples it is an open question what is a full description of the Poisson boundary, see for instance 
\cite[Question 5.8]{JMBMdlS}.

For the discrete affine group, we give a complete description of the Poisson boundary in terms of a space of configurations. 

\begin{theorem}\label{main}
For any finitely supported non-degenerate probability measure $\m$, the Poisson boundary of $(\DA(\T), \m)$ is the space of functions from $\T$ to $\mathfrak{S}_q$, endowed with the corresponding harmonic measures.
\end{theorem}

This description is in strong contrast with spread-out probability measures on the whole affine group $\Aff(\T)$, for which the Poisson boundary is the geometric boundary of the tree \cite{CKW}.

To briefly describe the origin of these configurations, observe that any element in the affine group can be described by an integer determining the action on horocycles (the vertical direction) together with a map from the vertex set of the tree to the permutation group $\mathfrak{S}_q$, called the portrait --- see Section \ref{afftree}. The discrete affine group consists precisely of the elements with portrait non-trivial only at finitely many vertices. 

Along a sample path of the random walk, the portraits stabilize pointwise to give a function from $\T$ to $\mathfrak{S}_q$. The configuration space consists of such portraits $\T\to \mathfrak{S}_q$. One key difference between this configuration space of $\DA(\T)$ and the lamp configuration space of a wreath product is that for $\DA(\T)$ the configurations are supported on a Schreier graph of the group, instead of the Cayley graph of the base group. Therefore we are lead to study the inverted orbits on $\T$. The reason for stabilization of configurations is that along the trajectory, the portrait is modified only at bounded distance from the inverted orbit of our chosen vertex, which is transient. This is similar to \cite{JMBMdlS} and \cite{KaiF}. The key point to get Theorem \ref{main} is a clear description of the inverted orbits --- see Section \ref{subsec:wordmetric}. The discrete affine group seems to provide the first non-trivial example of an action with well-understood inverted orbits, other than actions of groups on their Cayley graphs. 

It is interesting to compare the group $\DA(\T)$ with lamplighter groups. For instance, both admit a word metric described by a traveling salesman problem --- Proposition \ref{metric}. The group $\DA(\T)$ also resembles the lamplighter {\it graph} on $\T$ with its affine structure introduced by Sava in her thesis \cite{SavaThesis}, \cite{Sava}. The difference is that the moves on $\T$ are somewhat twisted by the action of $\DA(\T)$ on $\T$. 
From a quasi-isometric perspective, we can show that the discrete affine group embeds bi-Lipschitz into Sava's lamplighter graph --- see Section \ref{emb}. We do not know whether they are quasi-isometric.

The same argument shows that $\DA(\T)$ admits a bi-Lipschitz embedding into the wreath product between the free product of $q+1$ copies of $\Z_2$ (whose Cayley graph is a $q+1$-regular tree) and the permutation group $\mathfrak{S}_q$. 
Together with general results of Cornulier, Stadler and Valette \cite{CSV}, and Naor and Peres \cite{NPemb}, \cite{NPLp}, this implies:

\begin{proposition}\label{L_Pcomp} The group $\DA(\T)$ admits a bi-Lipschitz embedding into $L_1$, and for all $p \geq 1$, the $L_p$-compression exponent of the discrete affine group is:
\[
\a_p^\ast \left(\DA(\T)\right)=\max \left\{\frac{1}{p}, \frac{1}{2}\right\}.
\]
\end{proposition}

The $L_p$-compression exponent is defined in Section \ref{compexp}. The $L_1$-isoperimetric profile and the return probability, see Section \ref{return-profile} for definitions, are two other invariants of quasi-isometry that can be computed for $\DA(\T)$:

\begin{proposition}\label{iso-return}
Let $\m$ be a finitely supported non-degenerate symmetric probability measure on $\DA(\T)$.
Then the $L_1$-isoperimetric profile satisfies
\begin{equation*}
\L_{1, \DA(\T), \m}(v) \simeq \frac{1}{\log \log v},
\end{equation*}
and the return probability
\begin{equation*}
\m^{(2n)}(e) \simeq e^{-n/(\log n)^2}.
\end{equation*}
\end{proposition}

Finally, let us mention that the discrete affine group is an example of a locally-finite-by-$\mathbb Z$ group without Shalom's property $H_{\textrm{FD}}$ \cite{BZH_FD}.

The organization of this paper is the following:
In Section \ref{afftree}, we review general properties of the full affine group $\Aff(\T)$.
In Section \ref{dag}, we introduce the discrete affine group $\DA(\T)$, present some basic properties of the group,
and describe the metric structure of a Cayley graph.
In Section \ref{PB}, we derive our main Theorem \ref{main}. 
We  also show that two random walks with vertical drifts of different signs, have mutually singular harmonic in Section \ref{harmonic}.
In Section \ref{emb}, we give an embedding of $\DA(\T)$ into a wreath product group, and compute the $L_p$-compression.
In Section \ref{return-profile}, we compute the $L_1$-isoperimetric profile and the return probability of random walk with respect to any finitely supported symmetric measure.

\section{The affine group of a regular tree}\label{afftree}

\subsection{Generalities} We recall the background on the affine group of a regular tree. We refer to \cite{CKW} for a detailed exposition.

Let $\T=\T_{q+1}$ be the $q+1$-regular tree, equipped with the graph distance $d_\T$.
An end of the tree is an equivalence class of geodesic rays (semi-infinite rays without backtracking), where two geodesic rays are equivalent when their symmetric difference is finite.
We denote by $\partial \T$ the set of all ends.
Fix, once and for all, an end $\o$ and a distinguished vertex $o$ in $\T$.
We still denote by $\o$ the unique geodesic ray issuing from $o$ towards $\o$, and by $\o(n)$ the vertex on $\o$ at distance $n$ from $o$.
The {\it Busemann function} $\b: \T \to \Z$ is defined by 
$$
\b(x)=\lim_{n \to \infty}(d_\T(x, \o(n)) - d_\T(o, \o(n))),
$$
where the limit exists for each $x$, since the sequence stabilizes eventually.
The horocycle $H_m$ is defined as the set of vertices whose value of $\b$ is $m$.
Then one can regard the tree $\T$ as a phylogenic tree rooted at infinity $\o$, and the horocycle $H_m$ as the $m$-th generation.

The  group of graph automorphisms of $\T$ which fix $\o$ is called the {\it affine group} of $\T$ and denoted by $\Aff(\T)$.
(Changing the end $\o$ yields an isomorphic group as the tree is regular.)
Observe that each affine transformation sends one horocyle to the other and
this fact motivates to define a map from the group $\Aff(\T)$ to the integers as follows:
For each $\g \in \Aff(\T)$ and $x \in \T$, we define 
$\F(\g):=\b(x.\g)-\b(x)$.
Note that $\F(\g)$ does not depend on the choice of $x$, and thus the map $\F: \Aff(\T) \to \Z$ is, in fact, a homomorphism. 
Define the subgroup $\Hor(\T)$ of $\Aff(\T)$ as the kernel of $\F$.
The group $\Hor(\T)$ preserves each horocycle.

One can verify that for the topology induced by the pointwise convergence on $\T$, the group $\Hor(\T)$ is an inductive limit of compact groups
(the subgroups fixing a vertex on the geodesic ray $\o$). Moreover, the exact sequence 
$$
1 \to \Hor(\T) \to \Aff(\T) \to \Z \to 1
$$
splits.
Indeed, choosing some $\a$ such that $\F(\a)=1$ determines the semidirect product decomposition of $\Aff(\T)$ by $\Hor(\T) \rtimes \Z$.
The group $\Aff(\T)$ is a locally compact topological group which is non-unimodular by \cite{SW90} and amenable as a topological group by \cite{Neb}.

\subsection{Labelling of the tree} 
Let us introduce a labelling of the tree $\T$. We identify the horocycle $H_m=\b^{-1}(m)$ with the set of all sequences indexed by $]-\infty,m]\cap \Z=\rrbracket -\infty,m\rrbracket$, taking values in the alphabet $\mathcal{A}_q=\{0,\dots,q-1\}$, and constant equal to $0$ on a subinterval $\rrbracket -\infty,n\rrbracket$. In the phylogenic description, the $k$-ancestor of a sequence is obtained by deleting the last (rightmost) $k$ entries. The edges of the tree link each vertex to its first ancestor. 

Up to changing the labelling, we can choose $\o(n)$ to be the constant sequence equal to $0$ on $\rrbracket -\infty, -n \rrbracket$, and $\a$ to be the shift $(u_k)_{k \leq m} \mapsto (u_{k-1})_{k \leq m+1}$ from $H_m$ to $H_{m+1}$.

The labelling extends naturally to the boundary, parametrized by the set of sequences indexed by $\Z$ taking values in $\mathcal{A}_q$ and constant equal to $0$ on a neighborhood of $-\infty$, plus the empty sequence corresponding to our chosen end $\o$ --- see Figure \ref{fig_tree}.
We warn that the sign of index $n$ for $\o(n)$ and $H_n$ is opposite.

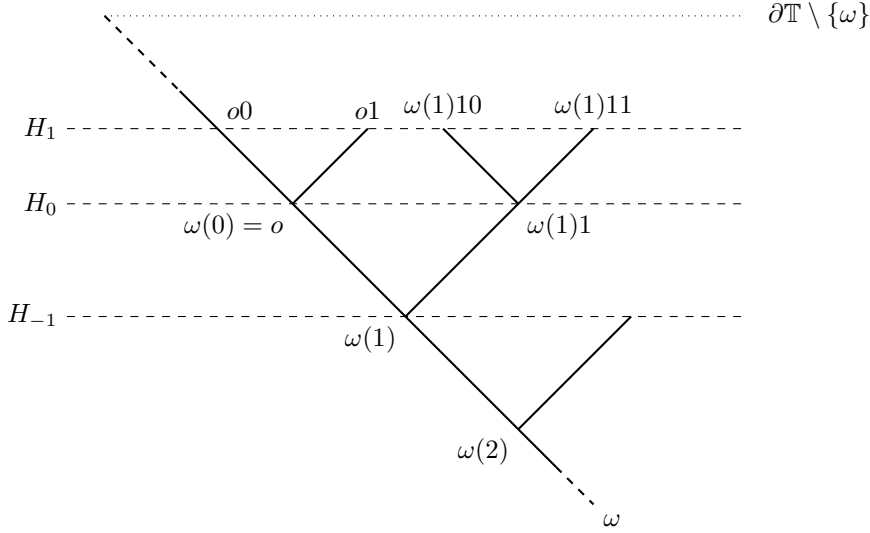
\begin{figure}
	\begin{center}
\begin{tikzpicture}

\draw[dotted] (0.5,3.5)--(9,3.5) ;
\node at (10,3.5) {$\partial \mathbb{T} \setminus\{\omega\}$};
\draw[thick] (2,2)--(1.5,2.5);
\draw[thick, dashed] (0.5,3.5)--(1.5,2.5);
\draw[thick] (2,2)--(6,-2);
\draw (2,2) node[above right]{$o0$};
\draw (4,2) node[above]{$o1$};
\draw (3,1) node[below left]{$\omega(0)=o$};
\draw[thick] (4.5,-0.5) --(7,2);
\draw (4.5,-0.5) node[below left]{$\omega(1)$};
\draw[thick] (5,2) --(6,1);
\draw (5,2) node[above]{$\omega(1)10$};
\draw (7,2) node[above]{$\omega(1)11$};
\draw (6,1) node[below right]{$\omega(1)1$};
\draw[thick] (3,1) --(4,2);
\draw[thick] (6,-2) --(7.5,-0.5);
\draw[thick] (6,-2) --(6.5,-2.5);
\draw (6,-2) node[below left]{$\omega(2)$};
\draw[thick, dashed] (6.5,-2.5)--(7,-3) node[below right]{$\omega$};
\draw[dashed] (0,2)--(9,2);
\draw (0,2) node[left]{$H_1$};
\draw[dashed] (0,1)--(9,1);
\draw (0,1) node[left]{$H_0$};
\draw[dashed] (0,-0.5)--(9,-0.5);
\draw (0,-0.5) node[left]{$H_{-1}$};

\end{tikzpicture}
	\end{center}
	\caption{The affine $3$-regular tree with its labelling.}
	\label{fig_tree}
\end{figure}

\subsection{Notion of portrait} 
If an element $\g$ of $\Hor(\T)$ fixes a vertex $v$, then it induces an automorphism of the rooted subtree $\T_v$ obtained by restriction to the root vertex $v$ and all its descendants. Automorphism groups of rooted trees have been largely studied --- see for instance \cite{GriNewH}, \cite{Nek}, \cite{Zuk} and references therein. The notion of portrait is a useful tool in their study, and can be generalized here as follows.

Consider the symmetric group $\mathfrak{S}_{q}$ acting by permutations on $\mathcal{A}_q$. 
Let $\g: \T \rightarrow \mathfrak{S}_{q}$ be a function denoted $v \mapsto \g[v]$. 
Assume that $\g[\o(n)]=e$ for all but finitely many $n$. 
Such a function defines an element of $\Hor(\T)$, still denoted $\g$, as follows.

For $n \ge N=\max\{n \ : \ \g[\o(n)]\neq e\}$, the action of $\g$ on the vertex $\o(n)$ is trivial. 
Now assume that we know the action of $\g$ on a vertex $v$ identified with its labelling sequence. 
The action of $\g$ on the descendant $vt$, where $t \in \mathcal{A}_q$, is given recursively by the ``automata rule'':
\begin{eqnarray}\label{automata}
vt.\g=(v.\g) (t.\g[v]),
\end{eqnarray}
where $t.\g[v]$ is the action of $\g[v]\in\mathfrak{S}_q$ on $t \in \mathcal{A}_q$. This permits to define by induction the action of $\g$ on all $\T$ since any vertex $v$ has the form $\o(n)t_1\dots t_k$ for some $n \ge N$ and $t_i$ in $\mathcal{A}_q$.

By construction, any element $\g \in \Hor(\T)$ can be obtained bijectively from such a function $\T \rightarrow \mathfrak{S}_q$ still denoted by $\g$ and called its {\it portrait}. 
The composition rule for the portrait of a product is:
\begin{eqnarray}\label{portrait-hor} 
\g\g'[v]=\g[v]\g'[v.\g]. 
\end{eqnarray}
We use right action, meaning $\g$ is applied first and then $\g'$. We refer to \cite{GriNewH} for details in the case of a rooted tree.

By the semidirect product decomposition $\Aff(\T)=\Hor(\T) \rtimes \Z$, an element of the affine group is identified with a pair $g=(\g,m)=\g \a^m$. The product rule is:
\begin{eqnarray*}\label{portrait-shift}gh=(\g, m)(\g', m')=(\g (\g')^{\a^m}, m+m'), \end{eqnarray*}
where the conjugate $\g^\a=\a\g\a^{-1}$ has the portrait $\g^\a[v]=\g[v.\a]$.

By abuse of notation, we denote by $g[v]$ instead of $\g[v]$ when $g=\g \a^m$, so for instance $\a[v]=e$ for all $v$. The combination of the previous line with (\ref{portrait-hor}) gives: for $g,h \in \Aff(\T), v \in \T$,
\begin{eqnarray}\label{portrait}
\quad gh[v]=g[v]h[v.g].
\end{eqnarray}
It follows that the portrait of a product $gh$ is obtained from that of $g$ after multiplying by $h[u]$ at position $u.g^{-1}$, for every $u$ in $\T$.

\section{The discrete affine group}\label{dag}

\subsection{Generalities} Since the subgroup of $\Hor(\T)$ consisting of elements with finitely supported portrait is invariant under conjugation by the shift, the following is a subgroup of $\Aff(\T)$, which we call the {\it discrete affine group}.
$$\DA(\T):=\{g \in \Aff(\T) \ : \ g[v]=e \textrm{ for all but finitely many }v\}. $$
This is neither a lattice of the affine group $\Aff(\T)$ (which cannot exist by the non-unimodularity),
nor a discrete subgroup of the topological affine group, but a finitely generated dense subgroup of $\Aff(\T)$ --- see Proposition \ref{gdeal}.

For each element $\s \in \mathfrak{S}_{q}$ and each vertex $v \in \T$, denote $\d^\s_v$ the horocyclic element with portrait given by
\[
\d_v^\s[w]=\left\{\begin{array}{ll} e &  \textrm{ if } w\neq v \\ \s  & \textrm{ if } w=v. \end{array}\right.
\]
For each $v \in \T$, the morphism $\s \mapsto \d^\s_v$ is injective of finite image $\d_v^{\mathfrak{S}_q}\subset\Hor(\T)$ isomorphic to $\mathfrak{S}_{q}$.

\begin{proposition}\label{gdeal} The group $\DA(\T)$ 
\begin{enumerate}
\item is generated by the finite set $S_0=\{\a,\a^{-1}\} \cup \d_o^{\mathfrak{S}_{q}}$,
\item is dense in $\Aff(\T)$ in the pointwise convergence topology,
\item is elementary amenable,
\item does not satisfy any group law.
\end{enumerate}
\end{proposition}

Recall that a group $G$ satisfies a {\it group law} if there exists a non-trivial irreducible word $w$ in an alphabet of size $k$ such that $w(g_1, \dots, g_k)=\textrm{id}$ for all $g_i \in G$, $1 \le i \le k$.

\begin{proof}
(1) As for all $m \in \Z$, we have $(\d_o^\s)^{\a^m}=\d_{\o(m)}^\s$, it is sufficient to check that $\d_{\o(m)}^{\mathfrak{S}_q}$ for $m$ in $\Z$ generates $\Hor(\T)\cap \DA(\T)$. It follows from the following folklore lemma applied to show that $\d_{\o(m)}^{\mathfrak{S}_q}$ for $m=-M$ to $M$ generate the subgroup of $\Hor(\T)$ consisting of elements with portrait trivial outside the $q$-ary subtree rooted at $\o(M)$ of depth $2M+1$ (which is naturally isomorphic to the group of automorphism of this finite rooted tree).

\begin{lemma}
Let $T(k)$ be a $q$-regular rooted tree of finite depth $k+1$, with a distinguished ray from the root to a leaf. For $0\leq i \leq k$, denote $\D_i \simeq S_q$ the subgroup of automorphisms of $T(k)$ with portrait supported on the $i$-th vertex of the ray. Then the union $\cup_{i=0}^k\D_i$ is a generating set of $\Aut(T(k))$.
\end{lemma}

\begin{proof}
We use induction on $k$ and the permutational wreath product isomorphism $\Aut(T(k)) \simeq \Aut(T(k-1)) \wr_{\{0,\dots,q-1\}} \D_0$. By induction, $\D_1,\dots,\D_k$ generate the first copy of $\Aut(T(k-1))$. The other copies are obtained by conjugation by $\D_0$.
\end{proof}

(2) We have to prove that given an element $g=(\g, m)$ in $\Aff(\T)$ and a finite collection of points in the tree, we can find an element in $\DA(\T)$ with the same action as $\g$ on these points. 
For this, choose $\o(n)$ a common ancestor of all the given points. Any element $\g'$ in $\Hor(\T)$ fixing $\o(n)$ and with the same portrait as $\g$ on a finite subtree containing $\o(n)$ and the given points will have the same action as the horocyclic part of $g$. Thus $\g'\a^m$ is the required element.

(3) This group is a cyclic extension of the subgroup of $\Hor(\T)$ the elements of which have finite support. 
The latter is locally finite, whence the group $\DA(\T)$ is elementary amenable.

(4) is obtained from the following theorem by Ab\'ert.

\begin{theorem}[Theorem 1 in \cite{A}]\label{Abert}
Let $G$ be a group acting on a set $X$, satisfying the following separability condition:
for every finite subset $Y$ of $X$, the pointwise stabilizer 
\[
G_Y:=\bigcap_{v \in Y}\{g \in G \ : \ gv=v\}
\]
does not stabilize any point outside $Y$.
Then $G$ does not satisfy any group law.
\end{theorem}

We observe that the action of $\DA(\T)$ on $\partial \T \setminus\{\o\}$ satisfies the separability condition in Theorem \ref{Abert}.
Indeed, it follows from the fact that for every $v \in \T$, we have $\d_{v}^{\mathfrak{S}_q} \subset \DA(\T)$.
\end{proof}

\begin{remark} We can replace $\mathfrak{S}_{q}$ by any of its subgroup $S$ with transitive action on $\mathcal{A}_q=\{0,\dots,q-1\}$, and obtain a group $\DA_S(\T)$ finitely generated, elementary amenable and dense in $\Aff_S(\T)$ in which all permutations of the portrait are in $S$ rather than arbitrary.
\end{remark}

\subsection{Description of the word metric}\label{subsec:wordmetric}

Our first aim here is to understand how to compute the portrait of a given word in the generators $S_0=\{\a,\a^{-1}\} \cup \d_o^{\mathfrak{S}_{q}}$ of the discrete affine group $\DA(\T)$.

By induction, let $g$ be such a word  and $h$ be a generator. By (\ref{portrait}), the portrait of $gh$ satisfies that 
for $v \in \T$, 
\begin{eqnarray}\label{step}\quad
gh[v]=g[v]h[v.g]=\left\{\begin{array}{ll}
g[v] & \textrm{ if } h\in \{\a,\a^{-1}\} \\
g[v]\d_o^\s[v.g]=g[v]\d_{o.g^{-1}}^\s[v] & \textrm{ if } h=\d_o^\s.
\end{array} \right.
\end{eqnarray}
In particular, the portrait is the same as that of $g$ except possibly at the point $o.g^{-1}$. This makes it important to understand what happens to this point when we multiply by a generator in $S_0$. So we want to express $o.(gh)^{-1}=o.h^{-1}g^{-1}$ in terms of $o.g^{-1}$ and the generator $h$. We record the

\begin{fact}\label{record}
We have the following multiplication rules:
\begin{enumerate}
\item If $h=\d_o^\s$, then $o.(\d_o^\s)^{-1}=o$, so $o.(g\d_o^\s)^{-1}=o.g^{-1}$.

\item If $h=\a$, then $o.\a^{-1}=\o(1)$ is the first ancestor of $o$, so $o.(g\a)^{-1}$ is the ancestor of $o.g^{-1}$.

\item If $h=\a^{-1}$, then $o.\a=o0$, so by the automata rule (\ref{automata})
\[
o.(g\a^{-1})^{-1}=o0.g^{-1}=(o.g^{-1})0.g^{-1}[o],
\]
which is a first descendant of $o.g^{-1}$ determined by $g^{-1}[o]=\left(g[o.g^{-1}] \right)^{-1}$.
\end{enumerate}
\end{fact}

Consider a word $w_n=x_1\dots x_n$ (which may not be reduced) in the generators $S_0$ and denote by $w_i=x_1\dots x_i$ its prefix of length $i$. It follows from (\ref{step}) that the support of the portrait of $w_n$ is included in the inverted orbit $O_n=\{o,o.w_1^{-1}, o.w_2^{-1},\dots,o.w_n^{-1}\}$. 

The inverted orbit $O_n$ is a sequence of points in the Schreier graph $\Gc(\DA(\T),S_0,o)$ --- see Figure \ref{Schreier}. Usually, it is extremely difficult to understand the inverted orbit of a word in a given action. In the present particular case, Fact \ref{record} permits to describe clearly $O_n$ from the word $w_n$. 
We make use of it to estimate the word metric in $\DA(\T)$. 

The set $S_1=\d_o^{\mathfrak{S}_{q}}\{\a^{\pm1}\}\d_o^{\mathfrak{S}_{q}}$ is also a generating set of $\DA(\T)$, as follows easily from Proposition \ref{gdeal}(1). 
It can be compared to the ``switch-walk-switch'' generating sets in lamplighter groups, for which the word metric is given by travelling salesman paths on the base graph. This is also the case here.

In a graph $G$, given a starting vertex $x$, an end vertex $y$ and a collection $A$ of vertices, we denote by $\textrm{TSP}_G(A;x,y)$ the length of the shortest path starting at $x$, visiting all vertices in $A$ and ending in $y$, called a travelling salesman path.

Let us denote by $\supp(g)=\{v\in \T \ : \ g[v] \neq e\}$ the support of the portrait of an element $g$ of $\DA(\T)$. We set $\ell(g)$ to be the number of edges of the smallest subtree $T(g)$ containing $\supp(g)\cup\{o,o.g^{-1}\}$. 

\begin{proposition}\label{metric}
The word metric in $\DA(\T)$ with respect to the generating set $S_1=\d_o^{\mathfrak{S}_{q}}\{\a^{\pm1}\}\d_o^{\mathfrak{S}_{q}}$ is given by
\[
 |g|_{S_1}={\rm TSP}_\T(\supp(g);o,o.g^{-1})=2\ell(g)-d_{\T}(o,o.g^{-1}).
\]
for all $g \in \DA(\T)\setminus \d_o^{\mathfrak{S}_{q}}$. Moreover $|\d_o^\s|_{S_1}=2$ for $\s$ non-trivial in $\mathfrak{S}_q$.
\end{proposition}

The second equality is obtained by the well-known description of travelling salesman paths in trees. We could also express the word metric with respect to $S_0$ but it would be slightly less elegant.

\begin{proof}
By Fact \ref{record}, the inverted orbit of a word in the generators $S_1$ is a path in the tree $\T$. By (\ref{step}) when we multiply by a generator in $S_1$, we modify the portrait at the two endpoints of the corresponding edge in the path. So the length of a representative word of $g$ is at least the length of a path solution to a travelling salesman problem in $\T$ starting in $o$ ending in $o.g^{-1}$ and visiting all vertices in $\supp(g)$. This gives the lower bound.

There remains to give a representative word of $g$ of right length. For this, let us first describe a path solution to the travelling salesman problem. The edges of this path are precisely those of the tree $T(g)$. The edges located on the geodesic from $o$ to $o.g^{-1}$ are covered exactly once, and those not on this geodesic are covered exactly twice. We can define a word $w_n$ in $S_1$ with inverted orbit precisely this path, and such that the portrait of the prefix $w_i$ coincides with the portrait of $g$ restricted to the vertices that have been visited for the last time before time $i$.

Indeed, assume $w_i$ is given, we can obtain $w_{i+1}$ as follows.

If time $i$ is the last visit of the path to vertex $o.w_i^{-1}$, multiply by an element in $\d_o^{\mathfrak{S}_{q}}$ that ensures $w_{i+1}[o.w_i^{-1}]=g[o.w_i^{-1}]$. Then multiply by $\a$ or $\a^{-1}$ to reach the next point of the path.

If time $i$ is not the last visit of the path to vertex $o.w_i^{-1}$, multiply by an element in $\d_o^{\mathfrak{S}_{q}}$ that ensures $o.w_{i+1}^{-1}$ is the next vertex of the path (this is possible by Fact \ref{record}). Then multiply by $\a$ or $\a^{-1}$ to reach the next point of the path.

At final time $n$, multiply by an element in $\d_o^{\mathfrak{S}_{q}}$ that ensures $w_{n}[o.w_n^{-1}]=g[o.w_n^{-1}]$.
\end{proof}

\subsection{The Schreier graph of the action on $\T$}

\begin{figure}
	\begin{center}
	\includegraphics[width=145mm]{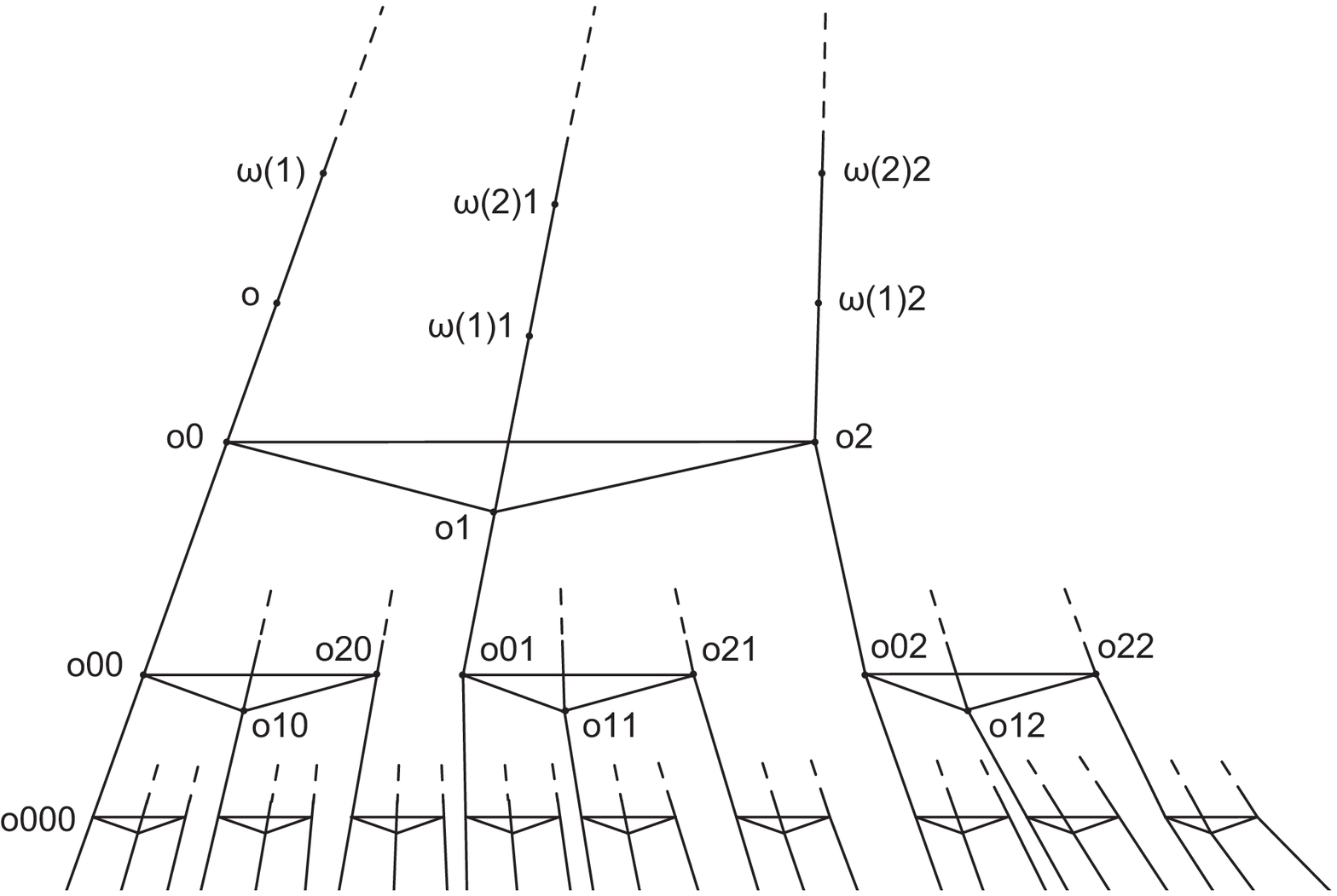}
	\end{center}
	\caption{Schreier graph $\Gc(\DA(\T), S_0, o)$ for $q=3$.}
	\label{Schreier}
\end{figure}

For future use, we describe the {\it Schreier graph} associated with the action of $\DA(\T)$ on $\T$ with respect to the generating set $S_0$ of Proposition \ref{gdeal} and the orbit of $o$. We denote it by $\Gc=\Gc(\DA(\T), S_0, o)$. The set of vertices is $\T=\{o.x \ : \ x \in \DA(\T)\}$, and the set of directed (labelled) edges is defined by $\{o.x, o.xs\}$ for some $s \in S_0$.
Both $\{o.x, o.xs\}$ and $\{o.xs, o.x\}$ are edges by the symmetry of $S_0$, and there are also loops. We draw a picture of $\Gc$ without loops and multiple edges in Figure \ref{Schreier}. We denote the graph distance by $d_\Gc$.

\begin{lemma}\label{distance}
For every $x \in \DA(\T)$, we have
$$
d_\T(o, o.x)\le d_\Gc(o, o.x) \le 2 d_\T(o, o.x).
$$
\end{lemma}

\begin{proof}
First note that in $\T$ we have $d_\T(o, o.x)=d_\T(o, o.x \s)$ for all $\s \in \d_o^{\mathfrak{S}_q}$.
For each $o.x \in \Gc$, let $w$ be a reduced word which realizes a directed path from $o$ to $o.x$ with $o.x=o.w$.
The number of $\{\a, \a^{-1}\}$ in $w$ bounds the number of steps from $o$ to $o.x$ in $\T$; this gives $d_\T(o, o.x) \le d_\Gc(o, o.x)$.

On the other hand, first let us consider those vertex $o.x$ in the subtree $\T_o$.
For a (unique) geodesic path from $o$ to $o.x$, each step in order to obtain $o.x$ we may apply a $\s \in \d_o^{\mathfrak{S}_q}$; 
hence $2$ times the number of steps gives the bound of $d_\Gc(o, o.x)$.
In general,
we note that a geodesic path from $o$ and $o.x$ in $\T$ is realized in such a way that there exists a subtree $\T_{o.\a^{-l}}$ for some $l \ge 0$, the path is a concatenation of the segment between $o$ and $o.\a^{-l}$, and the geodesic path between $o.\a^{-l}$ and $o.x$. 
The geodesic between $o.\a^{-l}$ and $o.x$ is the isomorphic image of the geodesic between $o$ and $o.x \a^l$ via $\a^{-l}$ in $\T$;
applying the argument used in above special case, we have $d_\Gc(o, o.x) \le 2 d_\T(o, o.x)$.
\end{proof}

\section{The Poisson boundary of the discrete affine group}\label{PB}

Let $\m$ be a probability measure on $\DA(\T)$. 
We consider the associated right random walk $w_n=x_1\dots x_n$, where the factors $x_i=\g_i\a^{m_i}$ form a sequence of independent random variables of law $\m$. We will always assume that $\m$ is {\it non-degenerate}, i.e.\ its support generates $\DA(\T)$ as a semigroup.

By Fact \ref{record}, understanding the random inverted orbit $O_n=\{o,o.w_1^{-1},\dots,o.w_n^{-1}\}$ is important to describe the behavior of the random walk. For a fixed $n$, the random word $w_n^{-1}=x_n^{-1}\dots x_1^{-1}=\check{w}_n$ coincides with the {\it left} random walk of law $\check{\m}$ given by $\check{\m}(g)=\m(g^{-1})$.

Both right and left random walks on the topological affine group of $\T$ have been studied by Cartwright, Kaimanovich and Woess \cite{CKW}. Note that as we use right actions on the tree, the roles of left and right random walks here are unfortunately exchanged  with their paper.

As already observed in \cite{CKW}, the inverted orbit $\{o.\check{w}_n\}$ is not a Markov process on $\T$, except in extremely particular case --- see Remark \ref{3.2LP}. On the contrary, the direct orbit $\{o.w_n\}$ is always a Markov chain, because the transitions depend only on space and not on time.

\subsection{Poisson boundary} We recall here basic facts about Poisson boundaries. We refer to \cite{KV}, \cite{Khyp}, \cite{Fur} or \cite[Chapter 14]{LPbook} for more informations.
The space of trajectories $\O:=\G^{\N}$ of a random walk of law $\m$ on a group $\G$ 
is endowed with 
the (left) diagonal action of the group $\G$ via each coordinate 
and
the probability measure $\Pr_x=\pi_\ast \m^{\N}$ obtained by pushing forward the Bernoulli measure $\m^{\N}$ on the space of increments by $\pi:(x_1,x_2,\dots)\mapsto(w_0, w_1,w_2,\dots)$ where $w_0=x$ corresponds to the starting point. 
Moreover the shift naturally acts on the space of trajectories by $\tau:(w_1,w_2,\dots)\mapsto(w_2,w_3,\dots)$. 
Note that the shift $\tau$ commutes with the group action.
Let $\Ic$ be the $\s$-field of shift-invariant events. 
The {\it Poisson boundary} of $(\G,\m)$ is the measurable space $(\O, \Ic)$ 
endowed with the family of probability measures $\{\Pr_x\}_{x \in \G}$.
Notice that arbitrary two measures $\Pr_x$ and $\Pr_y$ are mutually absolutely continuous since $\m$ is non-degenerate;
it follows that $L^\infty(\O, \Ic, \Pr_x)$ and $L^\infty(\O, \Ic, \Pr_y)$ are naturally isomorphic.
Let us denote by $\Pr:=\Pr_e$.

We call a function $f : \G \to \R$ is {\it $\m$-harmonic} 
if $\sum_{h \in \G} f(gh)\m(h)=f(g)$ for all $g$ in $\G$. 
Let $H^\infty(\G, \m)$ be the space of all bounded $\m$-harmonic functions on $\G$ endowed with $\ell^\infty$-norm.
One can check that the linear map from $L^\infty(\O, \Ic, \Pr)$ to $H^\infty(\G, \m)$ defined by
$$
\f \mapsto \int_{\O} \f(\o)d\Pr_x(\o),
$$
(where the right hand side is a function of $x$ in $\G$),
gives an isometric isomorphism as Banach spaces by the bounded martingale convergence theorem.

In practice, many explicit description of Poisson boundaries, e.g.\ \cite{KW}, \cite{BroAff}, \cite{BroSch}, \cite{Sava}, have been done using the ``Strip criterion'' due to Kaimanovich \cite{Khyp}. 
Indeed, the Strip criterion also works for a special random walk on $\DA(\T)$ (where one can adapt the methods in \cite{KW}).
In order to deal with a general random walk,
we use the following enhanced version of this criterion.

\begin{theorem}[Lyons-Peres Trap criterion \cite{LP} Corollary 2.3, and \cite{LPbook} Proposition 14.42]\label{Trap}
Let $\Ac$ be a $\G$-invariant sub-$\s$-field of $\Ic$. 
Assume that for every $\e>0$, there exists a random sequence of finite subsets $\{D_\ell\}_{\ell=1}^{\infty}$ of $\G$ and a constant $c>0$ such that
\begin{itemize}
\item[(i)] each $D_\ell$ is a (set-valued) measurable function with respect to $\Ac$,
\item[(ii)] almost surely there exists $N$ such that $| D_\ell| \le e^{\e \ell}$ for all $l\ge N$,
\item[(iii)] $\limsup_{\ell \to \infty}\Pr\left(\exists n \ge \ell, w_n \in D_\ell\right)\ge c$.
\end{itemize}
then $\Ac=\Ic$ modulo $\Pr$-null sets. In particular, $(\O, \Ac)$ gives a realization of the Poisson boundary.
\end{theorem}

The sets $D_l$ can be thought of as traps which can be defined in terms of $\Ac$ by (i), are small by (ii) and catch the random walk with positive probability by (iii).

\subsection{Description of the Poisson boundary: Main result}\label{section-Poisson}
Our aim here is to describe the Poisson boundary of $\DA(\T)$. 
Let us first assume that the step distribution $\m$ is finitely supported. 
Denote by $R$ the minimal integer such that for all $g \in \supp(\m)$,
\begin{eqnarray}\label{R}
\quad |\F(g)|\leq R, \quad \textrm{and} \quad g[v]=e \quad \textrm{when} \ d_\T(v,o)>R.
\end{eqnarray}
We assume that $\m$ is non-degenerate. 
In many cases we also assume that $\m$ is {\it aperiodic}, i.e.\ for any given group elements $g$ and $g'$ there exists a number $l \ge 1$ such that the support of the $l$-times convolution power of $\m$ contains both $g$ and $g'$.
If $\m$ is not aperiodic, then one may replace it by $\tilde{\m}:=(1/2)\m+(1/2)\d_e$, which is aperiodic and non-degenerate; 
moreover
$\m$-harmonic functions are always $\tilde{\m}$-harmonic functions and the converse is also true.

Cartwright, Kaimanovich and Woess observed that the sequence $\{o.\check{w}_n\}$ converges to an end $\partial\T$ almost surely. More precisely:

\begin{proposition}[A special case of Theorem 2 in \cite{CKW}]\label{CKWThm2}
Assume that $\m$ is non-degenerate on $\DA(\T)$, and consider the standard compactified topology in $\T\cup \partial \T$.
\begin{itemize}
\item [(i)] If $\E_{\check{\m}}\F >0$, then $\{o.\check{w}_n\}$ converges to a random element in $\partial\T\setminus\{\o\}$ almost surely.
\item [(ii)] If $\E_{\check{\m}}\F \le 0$, then $o.\check{w}_n \to \o$ almost surely.
\end{itemize}
\end{proposition}

As the group $\DA(\T)$ acts by automorphisms of the tree $\T$, we always have $d_\T(o,v)=d_\T(o.\check{w}_n,v.\check{w}_n)$. 
Combined with (\ref{portrait}) and (\ref{R}), It follows by (\ref{portrait}) and (\ref{R}) that for a fixed $v$ in $\T$ the sequence $w_n[v]$ changes values only when $d_\T(o.\check{w}_n,v)\leq R$, so is eventually constant by Proposition \ref{CKWThm2}. Denote $\g_\infty[v] \in \mathfrak{S}_q$ its limit.
This means that the sequence of portraits converges almost surely pointwise to a final configuration given by a Borel measurable map
$$\g_\infty :\Omega \rightarrow \prod_\T \mathfrak{S}_q.$$
Let $\Ac=\s(\g_\infty)$ be the sub-$\s$-field of $\Ic$ generated by final configurations. 
The map $\g_\infty$ is $\G$-equivariant with respect to a natural $\G$-action on $\prod_\T \mathfrak{S}_q$.
We will show that the hypothesis of Lyons-Peres Trap criterion stated in Theorem \ref{Trap} apply to $\Ac$ and deduce the following more precise version of Theorem \ref{main}.

\begin{theorem}\label{Poisson=config}
Assume that $\m$ is a finitely supported, non-degenerate probability measure on $\DA(\T)$. Then the space $\prod_\T \mathfrak{S}_q$ of final configurations with the distribution $\g_{\infty\ast}\Pr$ is a realization of the Poisson boundary of $(\DA(\T),\m)$.
\end{theorem}

This statement is completely analogous of what happens for lamplighter groups. The Trap criterion was designed in \cite{LP} to show that final configurations realize the Poisson boundary for random walks in lamplighter groups.

\begin{remark}\label{3.2LP} In her thesis \cite[Section 5.3]{SavaThesis}, Sava introduced some lamplighter random walks on the {\it graph} 
$\Z_2 \wr \T$ and identified their Poisson boundaries as spaces of lamp configurations. Random walks in the group $\DA(\T)$ resemble very much random walks in the graph $\Z_2 \wr \T$, as the portrait corresponds to the lamp configuration.

Sava's results were improved by Lyons and Peres \cite[Theorem 3.2]{LP}. Our proof will follow their ideas. A major difference in our case is that the random trajectory $\{o.\check{w}\}_n$ induced on $\T$ is usually {\it not} a Markov process. It has some memory encoded in the portrait. 

A particular case where this process is actually Markov is obtained when $\m$ is the equidistribution measure on the set $S_1=\d_o^{\mathfrak{S}_{q}}\{\a^{\pm1}\}\d_o^{\mathfrak{S}_{q}}$. Indeed we see by Fact \ref{record} that the entry $w_n[o.\check{w}_n]$ is completely randomized before each move by $\a^{-1}$. This is no longer true for general step distribution, even when $\m$ is symmetric, i.e., $\m=\check \m$.
\end{remark}

\subsection{Proof of Theorems \ref{main} and \ref{Poisson=config}}\label{mainproof}

The measure $\m$ on $\DA(\T)$ induces a random walk on $\Gc$ given by the forward orbit $\{o.w_n\}$. This is a Markov chain on $\Gc$ with transition probability $p(x, y)=\m(\{s: x.s=y\})$ for $x$, $y \in \Gc$.

\begin{lemma}\label{transience}
For $\m$ on $\DA(\T)$ as in Theorem \ref{Poisson=config}, the random walk induced by $\m$ on the Schreier graph $\Gc=\Gc(\DA(\T), S_0, o)$ is transient.

In particular, if $B_\Gc(o, R)$is the ball of radius $R$ centered at $o$ in $\Gc$, there exist $n_0$ and $c_R>0$ such that 
$$
\Pr(\forall n \ge n_0, o.w_n \notin B_\Gc(o, R))\ge c_R>0.
$$
\end{lemma}
\begin{proof}
We regard the Schreier graph $\Gc$ as a network endowed with conductance $1$ on each edge.
We use reduction of the network by contracting edges corresponding to $\s \in \d_o^{\mathfrak{S}_q}$, 
and observe that the reduced network contains a copy of $\T_o$, which is transient.
The Rayleigh monotonicity implies that the simple random walk on $\Gc$ is transient as well (e.g.\ \cite[Chapter 2]{LPbook}).
For the network endowed with transition probability $p(x, y)$, which is uniformly bounded away from $0$, but need not be reversible,
the random walk $\{o.w_n\}$ is still transient by \cite[Theorem 2.25 and Theorem 3.1]{Woess} . 
Here $n_0$ is the smallest number of steps to get out from the ball $B_\Gc(o, R)$.
\end{proof}

Assume $\E_{\check{\m}}\F>0$ (resp.\ $\E_{\check{\m}}\F\leq 0$), let us define $\tau_\ell$ for $\ell> 0$ as the first time when the random trajectory $o.\check{w}_n$ hits or crosses the horocycle $H_\ell$ (resp.\ $H_{-\ell}$). This coincides with the first time $\F(w_n)$ hits or crosses $\ell$ (resp.\ $-\ell$), thus is almost surely well defined. 

Removing the point $o.\check{w}_n$ splits $\T$ into $q+1$ subtrees. We denote $T_\ell$ the subtree containing $o$ and all other previous points of the inverted orbit. The following lemma asserts that the random trajectory $\{o.\check{w}_n\}$ does not come back to $T_\ell$ with positive probability. This is a quantitative form of transience.

\begin{lemma}\label{mainlemma}
There exists a constant $c$ depending only on $\m$ such that for all $\ell>0$,
\[
\quad\Pr(\forall n \geq \tau_\ell, o.\check{w}_n \notin T_\ell \ | \ \Fc_{\t_l}) \geq c>0,
\]
where $\Fc_{\t_l}$ is the $\s$-algebra generated by sequence $\{\check w_i\}_{i \le \t_l}$ up to $\t_l$.
In particular, for all $\ell>0$,
\[
\quad\Pr(\forall n \geq \tau_\ell, o.\check{w}_n \notin T_\ell) \geq c>0.
\]
\end{lemma}

This lemma is obvious when $\E_{\check{\m}}\F\neq 0$ for in this case we can use transience of the projected random walk $\b(o.\check{w}_n)=-\F(w_n)$ on the integers, which is finitely supported and drifted. 

\begin{proof}[Proof of Lemma \ref{mainlemma}]
Recall that $\t_\ell$ is the first time when the random trajectory $o.\check w_n$ hits or crosses the horocycle $H_{-\ell}$.
Then $\{\check w_{\t_\ell}w_{n+\t_\ell}\}_{n \ge 0}$ has the same distribution as  $\{w_n\}_{n \ge 0}$ by the strong Markov property. Let $R$ be as in (\ref{R}).
If $d_\Gc(o, o.\check w_{\t_\ell}w_{n+\t_\ell}) \ge 2R$, then $d_\T(o, o.\check w_{\t_\ell}w_{n+\t_\ell}) \ge R$ by Lemma \ref{distance}, 
and so $d_\T(o.\check w_{\t_\ell}, o.\check w_{n+\t_\ell}) \ge R$.
Therefore
if $o.w_n \notin B_\Gc(o, 2R)$ for all $n \ge \t_\ell$, then $o.\check{w}_n \notin B_\T(o,  R)$ for all $n \ge \t_\ell$.

In fact, the last condition implies that $o.\check{w}_n \notin T_\ell$ for all $n \ge \t_\ell$. Indeed, by definition of $R$, the condition implies that either all $o.\check{w}_n$ for $n \ge \t_\ell$ are in $T_\ell$, which cannot happen because the projection $\{\F(w_n)\}$ is recurrent on integers, or none.

By Lemma \ref{transience}, we obtain
\begin{eqnarray}\label{trans}
\Pr(\forall n \ge \t_\ell+n_0, o.\check{w}_n \notin T_\ell \ | \ \Fc_{\t_l}) 
\ge \Pr(\forall n \ge n_0, o.w_n \notin B_\Gc(o, 2R))\ge c_{2R}>0.			
\end{eqnarray}

To get Lemma \ref{mainlemma}, observe that for any $\ell'\geq \ell$, the probability that $o.\check{w}_n$ does not belong to $T_\ell$ for $\t_\ell\leq n \leq \t_{\ell'}$ is bounded below by a constant $c_1>0$. Choosing $\ell'=n_0(R+1)+\ell$ guarantees that $o.\check{w}_n$ does not belong to $T_\ell$ for $\t_{\ell'}\leq n \leq \t_{\ell'}+n_0$.
Lemma \ref{mainlemma} holds with $c=c_1c_{2R}>0$.
\end{proof}

\begin{lemma}\label{accumulation}
Assume that $\m$ is aperiodic.
The limit $\lim o.\check{w}_n\in \partial \T$ is almost surely the unique point of accumulation of $\supp(\g_\infty)$.
\end{lemma}

\begin{proof}
We give the proof in the case $\E_{\check{\m}}\F\leq 0$. 
Since $\m$ is non-degenerate and aperiodic,
one can find an integer $j$ such that $\m^{\otimes j}(v)$ and $\m^{\otimes j}(v')$ are both positive, where $v=\a^{R+1}$ and $v'=\s \a^{R+1}$ with $\s \in \d_o^{\mathfrak{S}_q}$ non trivial (in the case $\E_{\check{\m}}\F>0$, take negative powers of $\a$ instead) and $R$ defined in (\ref{R}). 

There almost surely exist infinitely many integers $n$ such that the random word $w_{n}$ ends with $v$ or $v'$ 
at $n=\t_\ell$,
where $\ell$ is a positive multiple of $R+1$. 
Moreover, by Lemma \ref{mainlemma}, almost surely there is an infinite subsequence $(\ell_k)_{k \geq 0}$ of this set of integers such that 
\begin{eqnarray}\label{no_return}
o.\check{w}_{n} \notin T_{\ell_k} \ \text{for all} \ n \geq \t_{\ell_k}.
\end{eqnarray} 
Indeed, for each $\ell$, the probability that (\ref{no_return}) does not hold is at most $1-c$,
so the probability that this does not happen for $n$ successive cases is at most $(1-c)^n$. 
This shows that $\supp(\g_\infty)$ has almost surely at most one accumulation point. 
There remains to show that $\supp(\g_\infty)$ is almost surely infinitely supported.

Now for each $k$, at least one of the two possibilities $v$ or $v'$ guarantees to have a non-trivial portrait at $o.\check{w}_{\t_{\ell_k}}$. 
This also holds for $\g_\infty$ by (\ref{no_return}) and this can be done independently for different $k$ as the $\ell_k$ differ by at least $R+1$.
Note that $v'$ appears infinitely many times, because the probabilities of occurrence of $v$ or $v'$ are given by independent Bernoulli laws of the same non-trivial parameter. So the size of $\supp(\g_\infty)$ is almost surely infinite. 
\end{proof}

\begin{proof}[Proof of Theorems \ref{main} and \ref{Poisson=config}]
In order to apply Theorem \ref{Trap}, define the random sequence of subsets $\{D_\ell\}_{\ell=1}^\infty$ of $\DA(\T)$ as follows. Let $c_1>0$ be a parameter to be fixed below. We denote $o_{c_1\ell}$ the point at distance $c_1\ell$ from $o$ on the unique geodesic ray from $o$ to $\lim o.\check{w}_n \in \partial \T$. Let $T_{c_1\ell}$ be the subtree of $\T\setminus\{o_{c_1\ell}\}$ containing $o$.

The subset $D_\ell$ consists of all elements $g=(\g,m)$ such that $c_1\ell \leq m \leq c_1\ell+R$ and 
\[
\g[v]=\left\{\begin{array}{ll}\g_\infty[v] & \textrm{ if } d_\T(v,o_{c_1\ell}) >R \textrm{ and } v \in T_{c_1\ell}, \\
e & \textrm{ if } d_\T(v,o_{c_1\ell}) >R \textrm{ and } v \notin T_{c_1\ell}.
 \end{array}\right.
\]
Clearly the size of $D_\ell$ is bounded by a constant depending only on $q$ and $R$, so point (ii) is satisfied. By Lemma \ref{accumulation}, the limit $\lim o.\check{w}_n$ is almost surely the unique point of accumulation of $\supp(\g_\infty)$, so the set $D_\ell$ is measurable with respect to $\Ac=\s(\g_\infty)$ and point (i) is satisfied.

Now let $\tau_{c_1\ell}$ as above be the first hitting or crossing time of $H_{c_1\ell}$ when $\E_{\check{\m}}\F>0$ (resp. $H_{-c_1\ell}$ when $\E_{\check{\m}}\F\leq0$). Also denote $o'_{c_1\ell}$ the point at the intersection of the horocycle $H_{c_1\ell}$ (or $H_{-c_1\ell}$) and the geodesic between $o$ and $o.\check{w}_{\tau_{c_1\ell}}$. Lemma \ref{mainlemma} asserts that with probability greater than $c>0$, we have $o'_{c_1\ell}=o_{c_1\ell}$ and $w_{\tau_{c_1\ell}}$ belongs to $D_\ell$. So point (iii) is satisfied when $c_1>R$ because $\tau_{c_1\ell}\geq c_1\ell/R$ by (\ref{R}).

Note that when $\E_{\check{\m}}\F\leq0$, we always have $o'_{c_1\ell}=o_{c_1\ell}=\o(c_1\ell)$.
\end{proof}

\begin{remark}\label{GKR}
Note that the existence of the subsequence $(\ell_k)$ in the proof of Lemma \ref{accumulation} provides another proof of Proposition \ref{CKWThm2}, where the end is given by the intersection of all $\T \setminus T_{\ell_k}$. The original proof used a general result on transience in non-unimodular groups by Guivarc'h, Keane and Roynette: as $\m$ is non-degenerate and $\DA(\T)$ is dense in $\Aff(\T)$ by Proposition \ref{gdeal}, the closed subgroup generated by the support of $\m$ is $\Aff(\T)$ which is non-unimodular. It follows by \cite[Theorem 51]{GKR} that the random walks $\{w_n\}$ and $\{\check{w}_n\}$ are transient in $\Aff(\T)$, i.e.\ they leave every compact set almost surely;
for details, see \cite[Theorem 2]{CKW}.
\end{remark}

\begin{remark}
In the case $\E_\m\F\neq0$, Theorem \ref{Poisson=config} still holds when $\m$ has a finite first moment. The proof above can be adapted as follows. Given $\e>0$, the sets $D_\ell$ now consist of all elements such that 
$c_1\ell \leq m \leq c_1\ell+\e\ell$ and 
\[
\g[v]=\left\{\begin{array}{ll}\g_\infty[v] & \textrm{ if } d_\T(v,o_{c_1\ell}) >\e\ell \textrm{ and } v \in T_{c_1\ell}, \\
e & \textrm{ if } d_\T(v,o_{c_1\ell}) >\e\ell \textrm{ and } v \notin T_{c_1\ell}.
 \end{array}\right.
\] 
The drift of the random walk induced on the integers still gives Lemma \ref{mainlemma} and point (i). The first moment condition ensures that the sequence of increments almost surely has a sublinear word norm. Combined with Proposition \ref{metric}, this shows that (\ref{R}) can be replaced by the almost sure existence of $N$ such that for all $n \geq N$, 
\[
|\F(x_n)|\leq \e n \quad \textrm{ and } \quad x_n[v]=e \textrm{ when }d_\T(v,o)>\e n.
\] 
The drift of the random walk induced on the integers still gives Lemma \ref{mainlemma} and so $w_{\t_{c_1\ell}}$ still belongs to $D_\ell$ with positive probability. Point (i) also still holds.

Finally, the law of large numbers gives a constant $c_2>0$ depending only on $\m$ such that $\t_\ell/\ell\rightarrow c_2$ almost surely. So point (iii) is satisfied when $c_1>1/c_2$.
\end{remark}

\begin{remark}\label{coupling}
We sketch here another proof of Lemma \ref{mainlemma} when $\E_\m\F=0$, considering only the random trajectory $o.\check{w}_n$ in the tree $\T$. Looking at times when $\F(w_n)$ first visits negative integers, we see that the line $\o=o.\a^\Z$ is recurrent. The metric projection $\proj(o.\check{w}_n)$ onto this line is a random process on $\Z$ with memory encoded in the portrait of $w_n$. However it is stochastically dominated by a random walk on $\Z$ drifted towards $-\infty$. The reason is that at times when $o.\check{w}_n$ is not on the line, the projection may only decrease, and at times when it is on the line, the projection may decrease by $x<0$ with probability $p(x)=\F_\ast\m(x)$ but it may increase by $x>0$ with probability $p(x)$ satisfying only $\sum_{x>0} xp(x)\leq \sum_{x>0}x\F_\ast\m(x)-c_1$ for some $c_1>0$ independent of the past. A precise coupling argument can be developed, but we prefer to omit the tedious details.
\end{remark}

\begin{remark}
The Schreier graph $\Gc$ shares some visual similarities with the Schreier graph of Thompson's group $F$ acting on dyadic rationals pictureded in \cite{Sav}. Both are essentially infinite binary rooted trees with infinitely many rays attached to it. Transience of this Schreier graph was used by Kaimanovich to show that the Poisson boundary of $F$ with respect to finitely supported measures is non-trivial \cite{KaiF}. In the discrete affine group, we can actually identify the Poisson boundary because the inverted orbit respects the geometric structure of $\T$, and this permits to define the traps. It is not known whether the $\m$-boundary of $F$ described in \cite{KaiF} is the whole Poisson boundary. Note that this boundary is actually trivial for some infinitely supported measures \cite{JZ}.
\end{remark}

\subsection{Description of the harmonic measure}\label{harmonic}

The aim here is to give a description of the distribution (the harmonic measure) $\g_{\infty\ast}\Pr$ on the space $\Pi_{\T}\mathfrak{S}_q$ of final configurations, which is the Poisson boundary by Theorem \ref{Poisson=config}.
We first describe a sublinear geodesic tracking property in the tree $\T$ when $\E_{\check \m}\F \neq 0$, in the spirit of Tiozzo \cite{Tiozzo}.

\begin{proposition}\label{geodesic-tracking}
Let $\m$ be as in Section \ref{section-Poisson}.
If $\E_{\check \m}\F \neq 0$,
then there exists a constant $C>0$ such that
for $\Pr$-almost every sample, there is a (random) geodesic ray $\x$ from $o$ converging to a point $\partial \T$ satisfying that
$$
\limsup_{n \to \infty}\frac{d_\T(o.\check{w}_n, \x)}{\log n} \le C.
$$
In the case when $\E_{\check \m}\F < 0$, the geodesic ray $\x$ is a deterministic one converging to $\o$,
while in the case when $\E_{\check \m}\F>0$, the geodesic ray $\x$ is a random one converging to some point in $\partial \T \setminus \{\o\}$.
\end{proposition}

Applying a general theorem by Tiozzo \cite[Theorem 6]{Tiozzo} immediately gives deviation of order $o(n)$ instead of $\log n$ for $\E_{\check \m}\F \neq 0$.

\begin{proof}
First we prove the claim when $\E_{\check \m}\F >0$.
Let $F:=\max_{g \in \supp \m} d_\T(o, o.g)$.
We define a sequence of regeneration time; namely, let
$$
\overline \t(n):=\inf \{k > 0 \ : \ \forall i \ge k, \F(\check w_i) > F n\}
$$
for $n \ge 0$.
We have $\overline \t(n)< \overline \t(n+1)$ for all $n$.
Note that $\overline \t(n)$ are not stopping times. 
Given the sequence $\{\overline \t(n) : n \ge 0\}$, 
a family of excursions $\{X_i^{(n)} : i \in I_n\}$ for $I_n:=[0, \overline{\t}(n+1) -\overline{\t}(n))$ defined by
\[
X_i^{(n)}:=\F(\check w_{i+\overline{\t}(n)}) - \F(\check w_{\overline{\t}(n)})
\] 
are independent,
and each excursion is an asymmetric random walk on $\Z$ drifted toward $+\infty$ conditioned on the event
$X_{\overline \t(n+1) - \overline \t(n) -1} \in (-\infty, F]$.
By comparison with the asymmetric random walk on $\Z$, 
we have that there exist constants $c_1, c_2>0$ such that for every $n \ge 0$ and for every $M \ge 0$,
$$
\Pr\left(\max_{i \in I_n} X_i^{(n)} \ge M \ | \ \{\overline \t(m) : m \ge 0\} \right) \le c_1 e^{-c_2 M}.
$$
Therefore given the sequence $\{\overline \t(n) : n \ge 0\}$, there exists a constant $C>0$ such that
$$
\max_{k \in [\overline \t(n), \overline \t(n+1))} d_\T(o.\check w_{\overline \t(n)}, o.\check w_k) \le C \log n,
$$
for all large enough $n$ almost surely.
Notice that the $F$-neighbourhood $\T(n, F)$ of the subtree $\T_{v_n}$ where $v_n:=o.\check w_{\overline \t(n)}$ are nested: 
$\T(n, F) \supset \T(n+1, F) \supset \cdots$,
and thus there exists a geodesic ray $\x$ from $o$ converging to some $\x$ in $\partial \T \setminus \{\o\}$, where we use the same symbol $\x$, such that
$d_\T(o.\check w_{\overline \t(n)}, \x) \le F$.
Clearly, we have $n \le \overline \t(n)$.
We have that $\Pr$-almost surely 
$$
d_\T(o.\check w_n, \g) \le C\log n + F
$$
for all large enough $n$, and obtain the claim.

Next we show the claim when $\E_{\check \m}\F < 0$.
We define the sequence of stopping times:
for $n \ge 0$,
$$
\t(n):=\min\{k \ge 0 \ : \ \F(\check w_k) \le -Fn\}.
$$
The family of excursions $\{X_i^{(n)} : i \in I_n\}$ for $n=0, 1, 2, \dots$ defined in a similar way to above 
are independent, and each excursion is an asymmetric random walk on $\Z$ drifted toward $-\infty$ up to a random stopping time.
Again comparison with the walk on $\Z$ yields the proof.
\end{proof}

Let us denote the harmonic measure by $\n_\m$ 
which is
$\g_{\infty\ast}\Pr$ on the space $\Pi_{\T}\mathfrak{S}_q$.
We show that $\n_\m$ have distinct supports depending on the signs of $\E_{\check \m}\F$.

\begin{proposition}\label{singular}
Let $\n_-$, $\n_0$ and $\n_+$ be the harmonic measures for aperiodic step distributions $\m_-$, $\m_0$, and $\m_+$ respectively satisfying $\E_{\check \m_-}\F<0$, $\E_{\check \m_0}\F=0$ and 
$\E_{\check \m_+}\F>0$.
Then $\n_-$, $\n_0$ and $\n_+$ have distinct supports in $\Pi_{\T}\mathfrak{S}_q$. In particular, they are mutually singular.
\end{proposition}

\begin{proof}
First we observe that,
by Lemma \ref{accumulation}, the support of final configurations $\supp \g_\infty$ has the accumulation point $\o$ in the case when $\E_{\check \m}\F \le 0$, and a (random) accumulation point in $\partial \T \setminus \{\o\}$ in the case when $\E_{\check \m}\F > 0$.
This shows that the support of $\n_+$ is distinct from that of $\n_-$ and $\n_0$
There remains to show that $\n_-$ and $\n_0$ have distinct supports.

Let us consider the infinite subtree $\T_{\ell}=\T_{\o(\ell)}$ rooted at $\o(\ell)$ and its finite subtree $\T_{\ell, f(\ell)}$ consisting of all descendants up to distance $f(\ell)$ from the root $\o(\ell)$.
We define a subset in $\T$ by
$$
D(\ell, f ):= \T_{\ell}\cup \bigcup_{m \ge \ell}\T_{m, f(m)},
$$
Let us define a set of configuration $
\hat D(f):=\left\{\f \in \Pi_{\T}\mathfrak{S}_q : \exists \ell, \ \supp \f \subset D(\ell, f)\right\}$. Proposition \ref{geodesic-tracking} gives a constant $C>0$ such that $\n_-(\hat D(C\log))=1$.

There remains to show that $\n_0(\hat D(C\log))=0$. In fact, for $f(\ell)=\ell+g(\ell)$ with $g(\ell)\leq \ell$, we claim that $\n_0(\hat D(f))=0$.

We argue as in the proof of Lemma \ref{accumulation}, where $v$ and $v'$ were defined. By properties of zero mean random walks on $\Z$, there are almost surely infinitely many $\ell_k$'s such that the trajectory $o.\check{w}_n$ reaches level $g(\ell_k)+1$ in the tree $\T_{\ell_k}$. Let $\t_k$ be the stopping time when $o.\check{w}_n$ first comes back to a level $\leq g(\ell_k)$. We consider only the almost surely infinitely many $k$'s where $w_{\t_k}$ ends with $v$ or $v'$. By Lemma \ref{mainlemma}, there is a fixed positive probability that the trajectory never visits again the subtree above vertex $o.\check{w}_{\t_k}$. At least one of the two possible endpoints $v$ or $v'$ guarantees that the final configuration is non trivial in $\T_{\ell_k}\setminus \T_{\ell_k,f(\ell_k)}$, proving the claim.
\end{proof}

\begin{remark}
Considering only the accumulation point of Lemma \ref{accumulation} gives a $\m$-boundary of the random walk (see \cite{Khyp},\cite{Fur} for a definition) which is a quotient of the Poisson boundary. In topological (spread-out measures) random walks on $\Aff(\T)$, this accumulation point is actually the Poisson boundary by \cite{CKW}. 

\end{remark}

\section{Embeddings into wreath products and $L_p$-compression}\label{emb}

\subsection{Embeddings into wreath products}
We consider the group $\DA(\T)$ and fix the set of generators $S_1$ as in Proposition \ref{metric},
and write the corresponding word metric by $d_{\DA}(x, y):=|x^{-1}y|_{\DA}$,
where we do not write it such as $|x^{-1}y|_{S_1}$ for clarity of notation.
We compare this metric in $\DA(\T)$ with a group of wreath product.

Given two groups $H$ and $F$, 
recall that  the wreath product $H \wr F$ is defined as $\oplus_{F} H \rtimes F$, where the semidirect product is given by
the left action: $\f \mapsto \f(x^{-1} \cdot )$ for $x \in F$ and $\f \in \oplus_{F} H$.

The set $\{(\d_e^{h}, e)\}_{h \in \mathfrak{S}_q} \cup \{({\bf 1}, s)\}_{s \in S_F}$, where $e$ is the identity element of the base group $F$, $\d_e^{h}$ is the element in $\oplus_F H$ given by $h$ at $e$ and the identity otherwise, and
${\bf 1}$ is the identity everywhere on $F$, generates $H\wr F$ and is often called the ``switch-or-walk" generating set. The set
\begin{eqnarray}\label{sws}
\{(\d_e^{h_1}, e)({\bf 1}, s)(\d_e^{h_2}, e) \ : \ h_1,h_2 \in H, s \in S_F\},
\end{eqnarray}
is another generating set of $H\wr F$ called ``switch-walk-switch'' generating set, more convenient for our purpose.

Let us consider the wreath product where $H=\mathfrak{S}_q$ and $F=F_{q+1}$ is the free product of $q+1$-copies of $\Z_2$.  The Cayley graph of $F_{q+1}$ with respect to the standard symmetric set of generators $S_F$ is a $q+1$-regular tree. We fix an identification of this tree with $\T$.
We consider the Cayley graph of $\G:=\mathfrak{S}_q \wr F_{q+1}$ associated with the set of generators in (\ref{sws}),
and denote the corresponding word metric by $d_\G(x, y):=|x^{-1} y|_\G$.

\begin{lemma}\label{biLip}
There exists an isometric embedding $f: \DA(\T) \to \G=\mathfrak{S}_q \wr F_{q+1}$, i.e.
for all $x, y$ in $\DA(\T)$, one has $d_\G \left(f(x), f(y)\right) = d_{\DA}(x, y)$.
\end{lemma}

\begin{proof}
We define the map
$f: \DA(\T) \to \G=\mathfrak{S}_q \wr F_{q+1}$ by
$$
g \mapsto \left(\{g[v]\}_{v \in \T}, o.g^{-1}\right)
$$
under the identification $F_{d+1}$ with $\T$.
It is well-known that the word metric $d_\G$ of wreath product with switch-walk-switch generating set is given by
$$
|(\f, x)|_\G=\left\{\begin{array}{ll} \TSP(\supp \f ; e, x), & \textrm{ if } (\f,x)\neq (\d_e^h,e) \\ 2 & \textrm{ if } (\f,x)= (\d_e^h,e) \textrm{ with } h\neq e \end{array} \right.
$$
where $\TSP(\supp \f ; e, x)$ denotes the minimal length of path in $\T$ from $e$ to $x$ visiting all the points in $\supp \f$. This is an isometry by Proposition \ref{metric}.
\end{proof}

\begin{remark}
The argument shows that there are bi-Lipschitz embeddings of $\DA(\T)$ into Sava's lamplighter graphs $\mathfrak{S}_q \wr \T$ and into the wreath product group $\mathfrak{S}_q \wr G$ as soon as the group $G$ contains an infinite bi-Lipschitz embedded binary tree. 

Of course, the group $\DA(\T)$ is not quasi-isometric to $\G$ because the latter is non-amenable. It would be interesting to show that it is not quasi-isometric to Sava's lamplighter graphs, nor to $\mathfrak{S}_q \wr (\Z_2 \wr \Z)$. 

\end{remark}

\subsection{$L_p$-compression of the discrete affine group}\label{compexp}
Recall that for a finitely generated group $\G$, and for $1 \le p < \infty$,
the {\it $L_p$-compression exponent} $\a^\ast_p(\G)$ is the supremum of all those $\a \ge 0$ for which there exists a Lipschitz map
$\G \to L_p$
such that for some constant $c>0$ and for all $x, y$ in $\G$,
$$
\|f(x)-f(y)\|_p \ge c d_\G(x, y)^\a,
$$
where $d_\G$ denotes a word metric in $\G$.
If $p=2$, then $\a^\ast_2(\G)$ is called the {\it Hilbert compression exponent}.
These are quasi-isometric group invariants introduced in \cite{GK}; for backgrounds and related results, see e.g.\ \cite{NPemb}, \cite{NPLp} and references therein.
For all $1 \le p < \infty$, we determine the exact values of $L_p$-compression exponents for $\DA(\T)$.

\begin{proposition}[Proposition \ref{L_Pcomp}]\label{Lp}
The group $\DA(\T)$ admits a bi-Lipschitz embedding into $L_1$.
Moreover,
for $1 \le p < \infty$,
$$
\a_p^\ast \left(\DA(\T)\right)=\max \left\{\frac{1}{p}, \frac{1}{2}\right\}.
$$
\end{proposition}

The proposition follows from several known results: the upper bound is obtained by the method using escape rate of random walks due to Naor and Peres \cite{NPLp}, and
the lower bound is derived from the corresponding result on wreath products over free groups by Cornulier, Stalder and Valette \cite{CSV}.

\begin{proof}[Proof of Proposition \ref{Lp}]
First we show the upper bound for $\a_p^\ast \left(\DA(\T)\right)$.
For a finitely generated group $G$,
let $\b^\ast(G)$ be the supremum of all those $\b \ge 0$ for which there exits a finite symmetric set of generators $S$, and a constant $c>0$ such that
for every $n \ge 1$,
$$
\E \left[d_G(e, w_n) \right]\ge c n^\b,
$$
where $\{w_n\}_{n=0}^\infty$ is the simple random walk on the Cayley graph of $G$ associated with $S$, starting at the identity $e$.
Naor and Peres show that
if $G$ is amenable, then
\begin{equation}\label{NP}
\a_p^\ast(G) \le \max\left\{\frac{1}{p}, \frac{1}{2}\right\}\frac{1}{\b^\ast (G)},
\end{equation}
see \cite[(5) in Section 1 and Theorem 1.3]{NPLp}.
In $\DA(\T)$, the simple random walk on any Cayley graph has a linear rate of escape since the Poisson boundary is non-trivial (Theorem \ref{Poisson=config}),
and thus we have $\b^\ast (\DA(\T))=1$; then (\ref{NP}) gives the desired upper bound for $\a_p^\ast \left(\DA(\T)\right)$.

Next we show the lower bound for $\a_p^\ast \left(\DA(\T)\right)$.
A result of Cornulier, Stalder and Valette implies that
$\a_1^\ast \left(\mathfrak{S}_q \wr F_{q+1}\right)=1$;
in fact, the group 
admits a bi-Lipschitz embedding into $L_1$ \cite[Proof of Proposition 7.2]{CSV}.
Combining with Lemma \ref{biLip}, we deduce that $\DA(\T)$ admits a bi-Lipschitz embedding into $L_1$,
in particular, $\a_1^\ast(\DA(\T))=1$.
In general, we have
$$
\a_p^\ast(\DA(\T)) \ge \max\left\{\frac{1}{p}, \frac{1}{2}\right\}\a_1^\ast(\DA(\T))
$$
by \cite[p.103]{NPLp}.
This yields the required lower bound for $\a_p^\ast(\DA(\T))$.
\end{proof}

\section{Return probability and isoperimetric profile}\label{return-profile}

For a finitely generated group $\G$ and for a symmetric probability measure $\m$ on the group $\G$,
 the {\it $L_1$-isoperimetric profile} $\L_{1, \G, \m}: [1, \infty) \to \R$ is defined by
$$
\L_{1, \G, \m}(v):=\inf\left\{\frac{1}{2}\sum_{x, y \in \G}|f(x)-f(y)|\m(y) \ : \ |\supp f| \le v, \|f\|_1=1\right\}.
$$
By using the co-area formula, the $L_1$-isoperimetric profile $\L_{1, \G, \m}$ is equivalent to
\begin{equation}\label{expansion}
\L_{1, \G, \m}(v)\simeq\inf\left\{\frac{1}{|U|}\sum_{x, y \in \G}1_U(x)1_{\G \setminus U}(xy)\m(y) \ : \ |U| \le v \right\},
\end{equation}
where given functions $f, g : (0, \infty) \to (0, \infty)$ (where the domain might be restricted on integers),
we write $f \lesssim g$ if there exist constants $C_1, C_2 >0$ such that $f(t) \le C_1 g(C_2 t)$ for all large enough $t$,
and write $f \simeq g$ if we have both $f \lesssim g$ and $f \gtrsim g$.

For a finite symmetric set of generators $S$ of a group $\G$, let $\m_S$ be the uniform measure on $S$.
For another finite symmetric set of generators $H$ of $\G$, we have $\L_{1, \G, \m_S} \simeq \L_{1, \G, \m_H}$.
In fact, one can check that the equivalence class (the {\it asymptotic type}) of $\L_{1, \G, \m_S}$ is a quasi-isometric invariant of the group $\G$.

The $L_1$-isoperimetric profile gives an estimate for the return probability $\m^{(2n)}(e)$, where $\m^{(2n)}$ denotes the $2n$-times convolution power of $\m$.
Again, for any two finite symmetric set of generators $S, H$ of a group $\G$, 
we have $\m_S^{(2n)}(e) \simeq \m_H^{(2n)}(e)$,
and in fact, the asymptotic type of $\m_S^{(2n)}(e)$ is a quasi-isometric invariant of the group $\G$ as shown by Pittet and Saloff-Coste \cite[Theorem 1.2]{PSC}.

\begin{proposition}[Proposition \ref{iso-return}]
Let $\m$ be a finitely supported symmetric probability measure on $\DA(\T)$.
Assume that the support of $\m$ generates the whole group as a semigroup.
Then
\begin{equation}\label{isoperimetric}
\L_{1, \DA(\T), \m}(v) \simeq \frac{1}{\log \log v},
\end{equation}
and
\begin{equation}\label{return}
\m^{(2n)}(e) \simeq e^{-n/(\log n)^2}.
\end{equation}
\end{proposition}

\begin{proof}
Given $r\in\mathbb{N}$, consider the set $L_{r}$ which
consists of vertices in the horocycle $H_{r}$ that are descendants
of $o$, that is those vertices with labeling $\ldots000.w$, $w\in\{0,\ldots,q-1\}^{r}$.
Let $C_{r}$ be the subgroup of elements of $\mbox{Hor}(\mathbb{T})$
such that the support of the portrait is contained in $L_{r}$,
\[
C_{r}=\{\gamma\in\mbox{Hor}(\mathbb{T}):\ \mbox{supp}\gamma\subseteq L_{r}\}.
\]
Then $C_{r}$ is isomorphic to the product $\mathfrak{S}_{q}^{|L_{r}|}$.
Consider the transition kernel that chooses one coordinate uniformly
from $L_{r}$ and update the permutation to be uniform in $\mathfrak{S}_{q}$,
that is 
\[
\eta_{r}=\frac{1}{|L_{r}|}\sum_{v\in L_{r}}\frac{1}{|\mathfrak{S}_{q}|}\sum_{\s \in\mathfrak{S}_{q}}\delta^{\s}_{v},
\]
where $\delta^{\s}_{v}$ is the indicator function on the element
$\gamma$ with portrait $\s$ at $v$ and trivial everywhere else.
We apply Erschler's isoperimetric inequality as in \cite[Proposition 3.1]{SCZ}, then we
have 
\[
\Lambda_{1,C_{r},\eta_{r}}\left(v\right)\ge\frac{1}{2K}\ \mbox{for all }v\le\frac{1}{K}\left|\mathfrak{S}_{q}\right|^{|L_{r}|/K}
\]
for some absolute constant $K>0$. 
Here for elements in the support
of $\eta_{r}$, the word distance relative to $S_0$ in $\DA(\T)$ is bounded by $2r+1$.  
Therefore by comparison, 
for the uniform measure $\m_{S_0}$ on the set of generators $S_0$ of $\DA(\T)$,
we have 
\[
\Lambda_{1,\DA(\T),\m_{S_0}}(v)\ge\frac{1}{2Cr}\ \mbox{for all }v\le\frac{1}{K}\left|\mathfrak{S}_{q}\right|^{|L_{r}|/K}=\frac{1}{K}(q!)^{q^{r}/K}.
\]
This gives 
\begin{equation}\label{isoperimetry-up}
\Lambda_{1,\DA(\T),\m_{S_0}}(v) \gtrsim 1/\log \log v.
\end{equation}
On the other hand,
for $r \in \N$,
consider the finite subset of $\DA(\T)$,
\[
U:=\left\{ \g \a^i \ : \ 0 \le i \le r, \supp\g \subseteq L'_r \right\},
\]
where $L'_r$ is the set of vertices descendant of $o$ belonging to $\cup_{0\leq s\leq r} H_s$. Then for $\m_{S_0}$,
we have 
\[
\L_{1, \DA(\T), \m_{S_0}}(v) \le \frac{2}{|S_0|r} \ \mbox{for all }v \ge |U|=r(q!)^{q^{r}},
\]
and this yields
$
\Lambda_{1,\DA(\T),\m_{S_0}}(v) \lesssim 1/\log \log v.
$
As we noted, $\Lambda_{1,\DA(\T),\m_{S_0}} \simeq \Lambda_{1,\DA(\T),\m}$
for a general finitely supported probability measure $\m$ that generates the whole group as a semigroup.

The lower bound (\ref{isoperimetry-up}) of $\L_{1, \DA(\T), \m}$ gives an upper bound of the return probability:
\[
\m^{(2n)}(e) \lesssim e^{-n/(\log n)^2},
\]
by using the Nash inequality \cite{C} (see also \cite{SCZ1}); we omit the detail.
For the return probability lower bound, 
we can use a simple argument that on the event that the projection to
$\mathbb{Z}$ is confined in $[-r,r]$, then the portrait is confined
to $B_{\mathbb{T}}(o,r)$. 
There exist constants $c_1, c_2>0$ such that for all large enough $r$ and $n$, one has
\[
\Pr\left(\max_{0 \le k \le n}|\Phi(\check w_k)| \le r \right) \ge c_1e^{-c_2 n/r^2}
\]
\cite[Lemma 1.2]{Alex}.
Then 
\[
\m^{(2n)}(e)\ge\frac{1}{2r(q!)^{\left|B_{\mathbb{T}}(o,r)\right|}}e^{-Cn/r^{2}}.
\]
Optimise choice of $r$ we obtain a matching lower bound. 
\end{proof}

\textbf{Acknowledgements.} J.B.\ is partially supported by project AGIRA ANR-16-CE40-0022.
R.T.\ is supported by JSPS Grant-in-Aid for Young Scientists (B) (No.\ 26800029).

\bibliographystyle{alpha}
\bibliography{DA}

\begin{thebibliography}{JMBMdlS16}

\bibitem[Ab{\'e}05]{A}
M.~Ab{\'e}rt.
\newblock Group laws and free subgroups in topological groups.
\newblock {\em Bulletin of the London Mathematical Society}, 37(04):525--534,
  2005.

\bibitem[Ale92]{Alex}
G.~Alexopoulos.
\newblock A lower estimate for central probabilities on polycyclic groups.
\newblock {\em Can. J. Math.}, 44(5):897--910, 1992.

\bibitem[AV14]{AmirViragDeg3}
G.~Amir and B.~Vir\'ag.
\newblock Positive speed for high-degree automaton groups.
\newblock {\em Groups Geom. Dyn.}, 8(1):23--38, 2014.

\bibitem[Bro04]{BroRenewal}
S.~Brofferio.
\newblock Renewal theory on the affine group of an oriented tree.
\newblock {\em Journal of Theoretical Probability}, 17(4):819--859, 2004.

\bibitem[Bro06]{BroAff}
S.~Brofferio.
\newblock The {P}oisson boundary of random rational affinities.
\newblock {\em Annales de l'institut Fourier}, 56(2):499--515, 2006.

\bibitem[BS11]{BroSch}
S.~Brofferio and B.~Schapira.
\newblock Poisson boundary of {$GL_d(\mathbb{Q})$}.
\newblock {\em Israel Journal of Mathematics}, 185(1):125--140, 2011.

\bibitem[BZ17]{BZH_FD}
J.~Brieussel and T.~Zheng.
\newblock Shalom's property ${H}_{{\rm F}{\rm D}}$ and extensions by
  $\mathbb{Z}$ of locally finite groups.
\newblock arXiv:1706.00707, 2017.

\bibitem[CKW94]{CKW}
D.~I. Cartwright, V.~A. Kaimanovich, and W.~Woess.
\newblock Random walks on the affine group of local fields and of homogeneous
  trees.
\newblock {\em Ann. Inst. Fourier, Grenoble}, 44(4):1243--1288, 1994.

\bibitem[Cou96]{C}
T.~Coulhon.
\newblock Ultracontractivity and {N}ash type inequality.
\newblock {\em Journal of Functional Analysis}, 141:510--539, 1996.

\bibitem[CSV12]{CSV}
Y.~Cornulier, Y.~Stalder, and A.~Valette.
\newblock Proper actions of wreath products and generalizations.
\newblock {\em Transactions of the American Mathematical Society},
  364(6):3159--3184, 2012.

\bibitem[Der75]{Der}
Y.~Derriennic.
\newblock Marche al\'eatoire sur le groupe libre et fronti\`ere de {M}artin.
\newblock {\em Z. Wahrscheinlichkeitstheor. Verw. Geb.}, 32:261--276, 1975.

\bibitem[Ers04]{ErschlerSubexp}
A.~Erschler.
\newblock Boundary behavior for groups of subexponential growth.
\newblock {\em Ann. of Math.}, 160:1183--1210, 2004.

\bibitem[Ers11]{Ers}
A.~Erschler.
\newblock Poisson-{F}urstenberg boundary of random walks on wreath products and
  free metabelian groups.
\newblock {\em Comment. Math. Helv.}, 86:113--143, 2011.

\bibitem[Fur71]{Furst}
H.~Furstenberg.
\newblock Random walks and discrete subgroups of {L}ie groups.
\newblock {\em Advances in probability and related topics}, 1:1--63, 1971.

\bibitem[Fur02]{Fur}
A.~Furman.
\newblock Random walks on groups and random transformations.
\newblock In North-Holland, editor, {\em Handbook of dynamical systems, Vol. 1a
  chap.12 Hasselblatt and Katok}, pages 931--1014, 2002.

\bibitem[GK04]{GK}
E.~Guentner and J.~Kaminker.
\newblock Exactness and uniform embeddability of discrete groups.
\newblock {\em J. London Math. Soc.}, 70(3):703--718, 2004.

\bibitem[GKR77]{GKR}
Y.~Guivarc'h, M.~Keane, and B.~Roynette.
\newblock {\em Marches Al\'eatoires sur les Groupes de Lie, Lecture Notes in
  Math., 624}.
\newblock Springer Berlin-Heidelberg-New York, 1977.

\bibitem[Gri00]{GriNewH}
R.~Grigorchuk.
\newblock Just infinite branch groups.
\newblock In {\em New Horizons in pro-p-Groups}, volume 184, pages 121--179.
  Progress in Mathematics, 2000.

\bibitem[JMBMdlS16]{JMBMdlS}
K.~Juschenko, N.~Matte~Bon, N.~Monod, and M.~de~la Salle.
\newblock Extensive amenability and an application to interval exchanges.
\newblock {\em Ergodic Theory Dynam. Systems}, online, 2016.

\bibitem[JZ16]{JZ}
K.~Juschenko and T.~Zheng.
\newblock Infinitely supported {L}iouville measures of {S}chreier graphs.
\newblock arXiv:1608.03554, 2016.

\bibitem[Kai91]{K}
V.~A. Kaimanovich.
\newblock Poisson boundaries of random walks on discrete solvable groups.
\newblock In {\em Probability measures on groups X (Oberwolfach 1990)}, pages
  205--238, New York, 1991. Penum.

\bibitem[Kai00]{Khyp}
V.~A. Kaimanovich.
\newblock The {P}oisson formula for groups with hyperbolic properties.
\newblock {\em Ann. of Math. (2)}, 152(3):659--692, 2000.

\bibitem[Kai16]{KaiF}
V.~A. Kaimanovich.
\newblock Thompson's group {$F$} is not {L}iouville.
\newblock arXiv:1602.02971, 2016.

\bibitem[KV83]{KV}
V.~A. Kaimanovich and A.~Vershik.
\newblock Random walks on discrete groups : boundary and entropy.
\newblock {\em Annals of Probability}, 11(3):457--490, 1983.

\bibitem[KW07]{KW}
A.~Karlsson and W.~Woess.
\newblock The {P}oisson boundary of lamplighter random walks on trees.
\newblock {\em Geom. Dedicata}, 124:95--107, 2007.

\bibitem[Led85]{Led}
F.~Ledrappier.
\newblock Poisson boundaries of discrete groups of matricies.
\newblock {\em Israel Journal of Mathematics}, 50:319--336, 1985.

\bibitem[LP15]{LP}
R.~Lyons and Y.~Peres.
\newblock Poisson boundaries of lamplighter groups: proof of the
  {K}aimanovich-{V}ershik conjecture.
\newblock {\em arXiv:1508.01845}, 2015.

\bibitem[LP16]{LPbook}
R.~Lyons and Y.~Peres.
\newblock {\em Probability on Trees and Networks}.
\newblock Cambridge University Press, 2016.

\bibitem[Neb88]{Neb}
C.~Nebbia.
\newblock Amenability and {K}unze-{S}tein property for groups acting on a tree.
\newblock {\em Pacific Journal of Mathematics}, 135(2):371--380, 1988.

\bibitem[Nek05]{Nek}
V.~Nekrashevych.
\newblock {\em Self-similar groups}.
\newblock Amer. Math. Soc., 2005.

\bibitem[NP08]{NPemb}
A.~Naor and Y.~Peres.
\newblock Embeddings of discrete groups and the speed of random walks.
\newblock {\em International Mathematics Research Notices}, 2008:rnn076, 2008.

\bibitem[NP11]{NPLp}
A.~Naor and Y.~Peres.
\newblock {$L_p$} compression, traveling salesmen, and stable walks.
\newblock {\em Duke Math. J}, 157(1):53--108, 2011.

\bibitem[PSC00]{PSC}
C.~Pittet and L.~Saloff-Coste.
\newblock On the stability of the behavior of random walks on groups.
\newblock {\em The Journal of Geometric Analysis}, 10(4):713--737, 2000.

\bibitem[Sav10a]{SavaThesis}
E.~Sava.
\newblock {\em Lamplighter Random Walks and Entropy-Sensitivity of Languages.
  Ph.D. thesis}.
\newblock Technische Universit{\"{a}}t Graz, 2010.

\bibitem[Sav10b]{Sava}
E.~Sava.
\newblock A note on the {P}oisson boundary of lamplighter random walks.
\newblock {\em Monatshefte f{\"u}r Mathematik}, 159(4):379--396, 2010.

\bibitem[Sav15]{Sav}
D.~Savchuk.
\newblock Schreier graphs of actions of {T}hompson's group {$F$} on the unit
  interval and on the cantor set.
\newblock {\em Geom. Dedicata}, 175:355--372, 2015.

\bibitem[SCZ15a]{SCZ}
L.~Saloff-Coste and T.~Zheng.
\newblock Isoperimetric profiles and random walks on some permutation wreath
  products.
\newblock {\em arXiv:1510.08830}, 2015.

\bibitem[SCZ15b]{SCZ1}
L.~Saloff-Coste and T.~Zheng.
\newblock Random walks and isoperimetric profiles under moment conditions.
\newblock {\em arXiv:1501.05929}, 2015.

\bibitem[SW90]{SW90}
P.~Soardi and W.~Woess.
\newblock Amenability, unimodularity, and the spectral radius of random walks
  on infinite graphs.
\newblock {\em Mathematische Zeitschrift}, 205:471--486, 1990.

\bibitem[Tio15]{Tiozzo}
G.~Tiozzo.
\newblock Sublinear deviation between geodesics and sample paths.
\newblock {\em Duke Math. J}, 164(3):511--539, 2015.

\bibitem[Woe00]{Woess}
W.~Woess.
\newblock {\em Random walks on infinite graphs and groups}, volume 138 of {\em
  Cambridge Tracts in Mathematics}.
\newblock Cambridge University Press, 2000.

\bibitem[Zuk08]{Zuk}
A.~Zuk.
\newblock Groupes engendr\'es par des automates.
\newblock {\em S\'eminaire Bourbaki, Ast\'erisque}, 311:141--174, 2008.

\end{thebibliography}

\textsc{\newline J\'er\'emie Brieussel --- Universit\'e de Montpellier 
} --- jeremie.brieussel@umontpellier.fr

\textsc{\newline Ryokichi Tanaka --- Tohoku University 
} --- rtanaka@m.tohoku.ac.jp

\textsc{\newline Tianyi Zheng --- UC San Diego
} --- tzheng2@math.ucsd.edu

\end{document}